\documentclass[12pt,a4paper,oneside]{amsart}
\usepackage[a4paper]{geometry}
\usepackage{mathptmx} 
\DeclareMathAlphabet{\mathcal}{OMS}{cmsy}{m}{n} 

\usepackage{perpage} 
\usepackage{amssymb,enumerate}

\usepackage[UglyObsolete]{diagrams} 
\usepackage[colorlinks=false]{hyperref}

\newtheorem{theorem}{Theorem}
\newtheorem*{proposition}{Proposition}
\newtheorem{lemma}{Lemma}
\newtheorem{conjecture}{Conjecture}
\newtheorem{question}{Question}

\MakePerPage[2]{footnote}

\newcommand{\vertbar}{\>|\>}

\DeclareMathOperator{\dcobound}{d}
\DeclareMathOperator{\Der}{Der}
\DeclareMathOperator{\End}{End}
\DeclareMathOperator{\id}{id}
\DeclareMathOperator{\im}{Im}
\DeclareMathOperator{\Ker}{Ker}

\begin{document}

\title{Special and exceptional mock-Lie algebras}
\author{Pasha Zusmanovich}
\address{
Department of Mathematics, University of Ostrava, Ostrava, Czech Republic
}
\email{pasha.zusmanovich@osu.cz}
\date{last revised January 1, 2017}
\thanks{Lin. Algebra Appl., to appear; arXiv:1608.05861}

\begin{abstract}
We observe several facts and make conjectures about commutative algebras
satisfying the Jacobi identity. The central question is which of those algebras
admit a faithful representation (i.e., in Lie parlance, satisfy the Ado theorem,
or, in Jordan parlance, are special).
\end{abstract}

\maketitle

\section*{Introduction}

A while ago, a new class of algebras emerged in the literature -- the so-called
mock-Lie algebras. These are commutative algebras satisfying the Jacobi 
identity. These algebras are locally nilpotent, so there are no nontrivial
simple objects. Nevertheless, they seem to have an interesting structure
theory which gives rise to interesting questions. And, after all, it is always 
curios to play with a classical notion by modifying it here and there and see 
what will happen -- in this case, to replace in Lie algebras anti-commutativity
by commutativity.

In \cite{agore-mil}, it was asked whether a fi\-ni\-te-di\-men\-si\-o\-nal 
mock-Lie algebra admits a fi\-ni\-te-di\-men\-si\-o\-nal faithful 
representation. In fact, an example providing a negative answer to this question
was given a while ago -- hidden in a somewhat obscure place, \cite{hjs}. We
provide an independent computer verification of that and similar examples, and 
examine several constructions and arguments for Lie algebras to see what works 
and what breaks in the mock-Lie case.

\section{Definitions. Preliminary facts and observations}\label{sec-def}

The standing assumption throughout the paper is that the base field $K$ is of 
characteristic $\ne 2,3$. An algebra $L$ over $K$, with multiplication denoted 
by $\circ$, is called \emph{mock-Lie} if it is commutative:
$$
x \circ y = y \circ x ,
$$
and satisfies the Jacobi identity:
$$
(x \circ y) \circ z + (z \circ x) \circ y + (y \circ z) \circ x = 0
$$
for any $x,y,z\in L$.

A substitution $x=y=z$ into the Jacobi identity yields 
$x^{\circ 3} = (x \circ x) \circ x = 0$. Conversely, linearizing the latter 
identity, we get back the Jacobi identity. Moreover, it is easy to see that 
assuming commutativity, the Jacobi identity is equivalent to the Jordan identity
$(x^{\circ 2} \circ y) \circ x = x^{\circ 2} \circ (y \circ x)$ (see, for 
example, \cite[Lemma 2.2]{burde-fial}). (On the other hand, commutative Leibniz
and commutative Zinbiel algebras form a narrower class of mock-Lie algebras, 
namely, commutative algebras of nilpotency index $3$: 
$(x \circ y) \circ z = 0$).

Thus, mock-Lie algebras can be characterized at least in the following four 
equivalent ways:
$$
\frac{\text{commutative}}{\text{Jordan}} \>\text{ algebras }\>
\frac{\text{satisfying the Jacobi identity}}{\text{of nil index } 3} .
$$

This class of algebras appeared in the literature under different names, 
reflecting, perhaps, the fact that it was considered from different viewpoints 
by different communities, sometimes not aware of each other's results. 
Apparently, for the first time these algebras appeared in 
\cite[\S 5]{zhevlakov}, where an example of infinite-dimensional solvable but 
not nilpotent mock-Lie algebra was given (reproduced in 
\cite[\S 4.1, Example 1]{nearly-ass}); further examples can be found in 
\cite[\S 4.1, Example 2 and \S 5.4, Exercise 4]{nearly-ass} and 
\cite[\S 2.5]{walcher}. In this and other Jordan-algebraic literature these 
algebras are called just ``Jordan algebras of nil index $3$''. In 
\cite{okubo-kamiya} they are called ``Lie-Jordan algebras'' (superalgebras are 
also considered there), and, finally, in the recent papers \cite{burde-fial} and
\cite{agore-mil} the term ``Jacobi--Jordan algebras'' was used. The term 
``mock-Lie'' comes from \cite[\S 3.9]{getzler-k}, where the corresponding operad
 appears in the list of quadratic cyclic operads with one generator. Despite 
that Getzler and Kapranov downplayed this class of algebras by (unjustly, in our
opinion) calling them ``pathological'' (basing on the fact that the mock-Lie operad is not Koszul), we
prefer to stick to the term ``mock-Lie''. These algebras live a dual life: as 
members of a very particular class of Jordan algebras, and as strange cousins of
Lie algebras.

Finitely generated mock-Lie algebras are nilpotent and finite-dimensional (see 
remark at the very end of \cite{zhevlakov}; in the class of finite-dimensional
algebras, or, more generally, algebras with maximal condition for subalgebras, 
this follows from the more general result about nilpotency of such Jordan nil-algebras, \cite[\S4.4]{nearly-ass}). The nilpotency index of 
finite-dimensional mock-Lie algebras can be arbitrarily large (take, for example, finitely 
generated subalgebras of either of the algebras presented in 
\cite[\S 4.1, Examples 1,2]{nearly-ass}), and is bounded by $3$ for 
$1$-generated algebras, by $5$ for $2$-generated algebras, by $9$ for 
$3$-generated algebras, and by $n+5$ for $n$-generated algebras, $n\ge 4$ 
(\cite[Main Theorem]{gutierrez}). Any (not necessarily finite-dimensional) 
mock-Lie algebra $L$ satisfies 
$(((L^2 \circ L^2) \circ L^2) \circ L^2) \circ L = 0$: in \cite[Lemma 3]{albert}
this is established via heavy computer calculations, and in 
\cite[Theorem 5]{sverchkov-prep} -- via heavy manipulations by hand with 
nonassociative identities. The same identity is also established in 
\cite[Theorem 1]{jacobs} via a straightforward computation using a computer 
algebra system Albert, though, strictly speaking, the latter cannot be counted 
as a rigorous (computer-assisted) proof, as the identity is being verified there
only in the case of characteristic $p=251$ of the base field. (Albert 
computations can be, however, turned into a rigorous proof by performing the 
same computation for different $p$'s and using a variant of the Chinese 
remainder theorem; also, nowadays one can compute in characteristic zero 
directly, see \S \ref{sec-nonsp} below). Mock-Lie algebras also satisfy the 
$3$-Engel identity $((x \circ y) \circ y) \circ y = 0$ (see, for example, 
\cite[Eq. (**) at p.~3268]{albert}).

Most of the basic definitions from the Lie structure theory are carried over
verbatim to the mock-Lie case. The center of a mock-Lie algebra $L$ is an ideal
consisting of elements $z$ such that $z \circ x = 0$ for any $x\in L$. The 
members of the lower central series are defined inductively as $L^1 = L$, and $L^n = L^{n-1} \circ L$ for $n>1$. The Jacobi identity implies that any 
monomial of length $n$, with arbitrary bracketing, lies in $L^n$. In particular,
$L^k \circ L^{n-k} \subseteq L^n$ for any $0 < k < n$. Thus, a mock-Lie algebra 
$L$ is nilpotent in any conceivable sense (general nonassociative, Jordan, etc.)
if and only if $L^n = 0$ for some $n$.

Representations of mock-Lie algebras may be approached in a universal way which
goes back to Eilenberg and is described nicely in \cite[Chapter II, \S 5]{J}: 
for an algebra $A$ from a given variety, a vector space $V$ is declared to carry
a structure of $A$-bimodule, if the corresponding semidirect sum $A \ltimes V$ 
(with $V$ carrying a trivial multiplication) belongs to the same variety. As 
mock-Lie algebras are commutative, the left and right actions of an algebra 
coincide, so we can speak about just modules. According to this approach, a 
vector space $V$ is a module over a mock-Lie algebra $L$, if there is a linear 
map (a representation) $\rho: L \to \End(V)$ such that 
\begin{equation}\label{eq-repr}
\rho(x \circ y)(v) = - \rho(x)(\rho(y)v) - \rho(y)(\rho(x)v)
\end{equation} 
for any $x,y\in L$ and $v\in V$. Representations of mock-Lie algebras are,
essentially, the same as homomorphisms to associative algebras: if $\rho$ is a
representation, then the map $-2\rho: L \to \End(V)^{(+)}$ is a homomorphism of
Jordan algebras, where the superscript $^{(+)}$ denotes, as usual, the passage 
from the underlying associative algebra to the Jordan algebra structure defined
on the same vector space via 
\begin{equation}\label{eq-circ}
a \circ b = \frac 12(ab + ba) .
\end{equation} 
As any associative algebra $A$ can be embedded into the algebra of all 
endomorphisms of a vector space, any homomorphism of Jordan algebras 
$L \to A^{(+)}$ can be prolonged to a homomorphism $L \to \End(V)^{(+)}$. 
Conversely, any such homomorphism, multiplied by $- \frac 12$, yields a map 
satisfying (\ref{eq-repr}). As in the Lie case, a representation is called 
\emph{faithful}, if its kernel is zero.

As finite-dimensional mock-Lie algebras are nilpotent, their irreducible 
representations are one-dimensional trivial. However, it seems to be interesting
to try to study finite-dimensional indecomposable representations, and to try
to classify finite-dimensional mock-Lie algebras into finite, tame, and wild 
representation types (such study for general Jordan algebras was initiated in 
\cite{kashuba-et-al}).

An entertaining fact (though not related to what follows): algebras over the 
operad Koszul dual to the mock-Lie operad can be characterized in three 
equivalent ways:
\begin{itemize}
\item anticommutative antiassociative algebras;
\item anticommutative $2$-Engel algebras;
\item anticommutative alternative algebras.
\end{itemize}

Here by antiassociative algebras we mean, following \cite{okubo-kamiya} and
\cite{markl-remm}, algebras satisfying the identity $(xy)z = -x(yz)$, and the 
$2$-Engel identity is $(xy)y = 0$.

Another entertaining fact (noted, for example, in \cite{okubo-kamiya}) is that 
mock-Lie algebras can be produced from antiassociative algebras the same way as 
they are produced from associative ones. Namely, given an antiassociative 
algebra $A$, the new algebra $A^{(+)}$ with multiplication given by 
``anticommutator'' (\ref{eq-circ}), is a mock-Lie algebra.

As the tensor product of algebras over Koszul dual operads form a Lie algebra
under the usual commutator bracket, the following question seems to be natural.

\begin{question}
Which ``interesting'' Lie algebras can be represented as the tensor product of
a mock-Lie algebra and an anticommutative antiassociative algebra?
\end{question}

As was noted, for example, in \cite[\S 3.9(d)]{getzler-k}, the mock-Lie operad 
is not Koszul, and hence does not admit a standard (co)homology theory (like,
for example, Lie or associative algebras). This however, does not preclude that
cohomology can be constructed in some ad-hoc, nonstandard, manner. 

In the class of mock-Lie algebras, when extending an algebra $L$ by a linear map
$D: L \to L$ to a semidirect sum $L \ltimes KD$, a role analogous to derivations
in the Lie case is played by antiderivations. Indeed, a necessary condition for
such semidirect sum to be a mock-Lie algebra, is $D$ to be an antiderivation of 
$L$. (This condition is not sufficient: in addition, we should impose conditions
following from the Jacobi identity involving two and three $D$'s; in particular,
since every mock-Lie algebra is $3$-Engel, we have $D^3 = 0$).
Recall that an \emph{antiderivation} of a mock-Lie algebra $L$ is a linear map 
$D: L \to L$ such that $D(x \circ y) = - D(x) \circ y - x \circ D(y)$ for any 
$x,y\in L$. Antiderivations, and, more generally, $\delta$-derivations (where 
$-1$ is replaced by an arbitrary $\delta \in K$) for various classes of algebras
were studied in a number of papers (see, for example, \cite{kayg} with a 
transitive closure of references therein). 

More generally, given a linear map $D: L \to V$ from a mock-Lie algebra $L$ into
an $L$-module $V$ (with the corresponding representation being denoted by 
$\rho$), the vector space $V \dotplus KD$ (here and below the symbol $\dotplus$
is reserved for the direct sum of vector spaces, while $\oplus$ denotes the 
direct sum of algebras) carries an $L$-module structure if and only if $D$ is an antiderivation of $L$ with values in $V$, i.e., 
$$
D(x \circ y) = - \rho(y)D(x) - \rho(x)D(y)
$$
for any $x,y\in L$. The space of all such antiderivations is denoted by 
$\Der_{-1}(L,V)$. The \emph{inner} antiderivations are antiderivations of the 
form $x \mapsto \rho(x)v$ for a fixed $v \in V$.

Interpreting in the standard way outer antiderivations (i.e., the quotient of 
antiderivations by inner antiderivations) as the 1st degree cohomology, and 
writing deformations of modules over mock Lie algebras as power series (\`a la 
Gerstenhaber), what also gives us the idea how the 1st degree cohomology should
look like, we get what might be the beginning of a complex responsible for 
cohomology of a mock-Lie algebra $L$ with coefficients in an $L$-module $V$:
$$
0 \to S^0(L,V) \overset{\dcobound}\to S^1(L,V) \overset{\dcobound}\to S^2(L,V) ,
$$
where $S^n(L,V)$ is the vector space of symmetric multilinear $n$-ary maps 
$L \times \dots \times L \to V$, and the formulas for the differential are:
\begin{align}\label{eq-h1}
\begin{split}
&\dcobound \varphi(x)\phantom{,y,z} = \rho(x) \varphi \\
&\dcobound \varphi(x,y)\phantom{,z} = \varphi(xy) + \rho(x)\varphi(y) + \rho(y)\varphi(x) . \phantom{aaaaaaaaaaaaaaaaaaaaaaaaaaaaaaa}. 
\end{split}
\end{align}

Similarly, writing a deformation of a mock-Lie algebra itself as a formal power
series, we get the part of the complex responsible for the second cohomology:
$$
S^1(L,V) \overset{\dcobound}\to S^2(L,V) \overset{\dcobound}\to S^3(L,V) ,
$$
where
\begin{align}\label{eq-h2}
\begin{split}
&\dcobound \varphi(x,y)\phantom{,z} = \varphi(xy) - \rho(x)\varphi(y) - \rho(y)\varphi(x) 
\\
&\dcobound \varphi(x,y,z) = \varphi(xy,z) + \varphi(zy,x) + \varphi(xz,y) +
\rho(z)\varphi(x,y) + \rho(x)\varphi(z,y) + \rho(y)\varphi(x,z) .
\end{split}
\end{align}

However, a (naive) idea to modify the Chevalley--Eilenberg complex for Lie 
algebra cohomology by injecting appropriately some signs does not work: it is 
not clear how to mangle the two definitions above, for the $1$st and $2$nd 
degree cohomology, together (note the difference in signs in formulas for 
$\dcobound\varphi(x,y)$ in (\ref{eq-h1}) and (\ref{eq-h2})). We may try to 
pursue a more modest goal and to construct just cohomology with trivial 
coefficients; then all terms containing $\rho$ in the formulas 
(\ref{eq-h1})--(\ref{eq-h2}) vanish, and together they give cohomology up to degree $3$; nevertheless, it is still not clear how to 
proceed in higher degrees. (The only sensible cohomology can be constructed in 
this way in characteristic $2$, what is an entirely different story and will be
treated in a separate paper). 

Note also that according to \cite{getzler-k}, an analog of cyclic (co)homology 
of mock-Lie algebras may be constructed, but it is not clear what its utility 
might be.

\begin{question}
Do mock-Lie algebras admit a ``good'' cohomology theory?
\end{question}

Another natural question is whether a mock-Lie algebra $L$ admits not a merely 
homomorphism to an associative algebra, but an embedding, or, what is 
equivalent, whether $L$ admits a faithful representation. (It is well-known that
not every Jordan algebra admits such an embedding; algebras, which do, are 
called \emph{special}; algebras, which do not, are called \emph{exceptional}). 
Note that since an analog of the universal enveloping algebra for 
finite-dimensional Jordan algebras is finite-dimensional, for finite-dimensional
algebras this is equivalent to existence of a \emph{finite-dimensional} faithful
representation, i.e., the validity of the analog of the Ado theorem for Lie 
algebras. 

A similar question about embedding of any mock-Lie algebra into an algebra of 
the form $A^{(+)}$ for an \emph{anti}associative algebra $A$ admits, as noted in
\cite{okubo-kamiya}, a trivial negative answer. Indeed, it is easy to see that in any antiassociative algebra $A$, and hence in any mock-Lie 
algebra of the form $A^{(+)}$, a product of any $4$ elements (under arbitrary
bracketing) vanishes. In particular, any mock-Lie algebra which is embeddable 
into an algebra of the form $A^{(+)}$ for an antiassociative algebra $A$, is 
necessarily nilpotent of index $4$.

We discuss properties of universal enveloping algebras of mock-Lie algebras in 
\S \ref{sec-uj}; in particular, we see how and why the arguments used for Lie 
algebras (notably, the Poincar\'e--Birkhoff--Witt theorem) fail, and establish speciality of mock-Lie algebras of small dimension.
In \S \ref{sec-lieado} we see how and why another chain of arguments, used in
\cite{ado} to establish the Ado theorem not utilizing universal enveloping 
algebras, fails in the mock-Lie case. An example showing that, in general, mock-Lie algebras are not special, was 
given in \cite{hjs}. That and similar examples are discussed in 
\S \ref{sec-nonsp}.

\section{Universal enveloping algebras}\label{sec-uj}

For a Jordan algebra $J$, a natural analog $U(J)$ of the universal enveloping 
algebra is defined as the quotient of the tensor algebra $T(J)$ (= free 
associative algebra freely generated by a basis of $J$) by the ideal generated
by elements of the form 
\begin{equation}\label{eq-id}
\frac 12(x \otimes y + y \otimes x) - x \circ y
\end{equation}
for $x,y \in J$. The restriction to $J$ of the natural homomorphism 
$T(J) \to U(J)$ yields a homomorphism of Jordan algebras 
$\iota: J \to U(J)^{(+)}$. The standard argument shows then that indeed, $U(J)$
is a universal object in the category of such homomorphisms, i.e., for any 
associative algebra $A$ and homomorphism of Jordan algebras 
$\varphi: J \to A^{(+)}$, there is a unique homomorphism of associative algebras
$\psi: U(J) \to A$ such that the diagram
\begin{diagram}
J & \rTo^\varphi & A          \\
  & \rdTo_\iota  & \uTo_\psi  \\ 
  &              & U(J) 
\end{diagram}
commutes. See \cite[Chapter II, \S 1]{J} for further details.

Specializing this construction to the mock-Lie case, one might expect that it 
will play a role similar to the universal enveloping algebra in the Lie case. 
For example, if $L$ is \emph{abelian}, i.e., with trivial multiplication, then
$U(L) = \bigwedge(L)$, the exterior algebra on the vector space $L$. However, 
the major property of Lie-algebraic universal enveloping algebras -- the 
Poincar\'e--Birkhoff--Witt theorem -- fails miserably in the mock-Lie case. The 
shortest and, arguably, the most elegant proof of the 
Poincar\'e--Birkhoff--Witt theorem uses the Gr\"obner bases technique: one 
proves that the set of defining relations of the universal enveloping algebra of
a Lie algebra is closed with respect to compositions and hence forms a Gr\"obner
base, and then the result follows from the composition lemma (see, for example,
\cite[\S 2.6, Example 4]{ufn}; we also follow the terminology adopted there). It
is edifying to look at this proof, and see how it breaks in the mock-Lie case.

If $L$ is a mock-Lie algebra with a basis $\{x_1, \dots, x_n\}$, $U(L)$ is an 
associative algebra with generators $\{x_i\}$ and relations 
\begin{align}\label{eq-rel}
\begin{split}
&x_ix_j + x_jx_i - 2x_i \circ x_j, \quad i<j  \\
&x_i^2 - x_i \circ x_i .                        
\end{split}
\end{align}

If one assumes the lexicographic order on monomials induced by 
$x_1 > x_2 > \dots $, and tries to compute compositions between these relations,
one arrives, after a series of simple compositions and reductions as in the Lie
case, to expressions of the form
\begin{equation}\label{eq-c}
x_i(x_j \circ x_k) + x_j(x_i \circ x_k) + x_k(x_i \circ x_j) .
\end{equation}

Depending on the multiplication table of $L$, further reduction of this 
expression with respect to the defining relations may be possible. In general, 
however, there is no reason why the expression (\ref{eq-c}) should be further 
reduced to zero, and thus, unlike in the Lie case, the relations defining the 
universal enveloping algebra are not closed with respect to composition. Of course, this still does not prove that 
the monomials 
\begin{equation*}
x_{i_1} \dots x_{i_k}, \quad 1 \le i_1 < \dots < i_k \le n
\end{equation*}
cannot form a basis of $U(L)$, but this is a strong
indication of the failure of the Poincar\'e--Birkhoff--Witt theorem. 

\begin{conjecture}\label{conj-pbw}
The mock-Lie algebras for which the Poincar\'e--Birkhoff--Witt theorem holds, 
are exactly the abelian algebras.
\end{conjecture}

For finite-dimensional algebras, the conjecture holds true. 

\begin{theorem}[I. Shestakov]\label{th-pbw}
The finite-dimensional mock-Lie algebras for which the 
Poincar\'e--Birkhoff--Witt theorem holds, are exactly abelian algebras.
\end{theorem}

The validity of the Poin\-ca\-r\'e--Birk\-hoff--Witt theorem for a mock-Lie 
algebra $L$ implies that $\dim U(L) = \dim \bigwedge(L) = 2^{\dim L}$. We will 
establish Theorem \ref{th-pbw} by providing two different proofs of the fact 
that for a finite-dimensional nonabelian mock-Lie algebra $L$, 
$\dim U(L) < 2^{\dim L}$. 

For the first proof, we need two simple lemmas.

\begin{lemma}\label{lemma-ul}
Let $L$ be a mock-Lie algebra, and $I$ its ideal. Then 
$\dim U(L) \le \dim U(L/I) \cdot \dim U(I)$.
\end{lemma}

\begin{proof}
Every element in $U(L)$ can be represented as the sum of products $XY$, where 
$X$ is a monomial of the form $x_{i_1} \dots x_{i_k}$, $x$'s belong to 
representatives of the cosets in $L/I$, and $Y$ is a monomial of the form 
$y_{j_1} \dots y_{j_{\ell}}$, $y$'s belong to $I$. The elements $X$ are subject
to the same relations (\ref{eq-rel}) modulo $I$, and elements $Y$ are subject to
relations (\ref{eq-rel}), with $x$'s replaced by $y$'s (of course, there could
be additional relations in $U(L)$ between $x$'s and $y$'s, but we do not care 
here about them). Thus all $X$'s linearly span $U(L/I)$, all $Y$'s linearly span
$U(I)$, and the statement of the lemma follows.
\end{proof}

\begin{lemma}\label{lemma-ab}
A nonabelian mock-Lie algebra such that all its proper ideals and all its proper
quotients are abelian, is isomorphic to one of the following algebras:
\begin{enumerate}[\upshape(i)]
\item $2$-dimensional algebra $\langle a,b \vertbar a^2 = b \rangle$;
\item $3$-dimensional algebra $\langle a,b,c \vertbar ab = c \rangle$;
\item 
$3$-dimensional algebra $\langle a,b,c \vertbar a^2 = c;\> ab = c \rangle$.
\end{enumerate}
\end{lemma}

In the multiplication tables between basic elements of algebras, we specify the
nonzero products only. Mock-Lie algebras of low dimension over an algebraically
closed field were described in \cite{burde-fial}. Algebra from heading (i) is 
denoted as $A_{12}$ in this classification, and is the only nonabelian 
$2$-dimensional algebra. Algebra from heading (ii), the 
``commutative Heisenberg algebra'', is isomorphic to the algebra $A_{13}$ (see 
\cite[Remark 3.3]{burde-fial}), and algebra from heading (iii) is isomorphic to 
$A_{12} \oplus A_{01}$, the direct sum of algebra from heading (i) and 
one-dimensional abelian algebra.

\begin{proof}[Proof of Lemma \ref{lemma-ab}]
Let $L$ be a mock-Lie algebra as specified in the condition of the lemma.
Since $L^2$ is a proper ideal of $L$, it can be enlarged to an ideal $I$ of 
codimension one, which is abelian. Write $L = I \dotplus Ka$ for some element 
$a \notin I$, and $a \circ x = f(x)$ for $x\in I$, where $f: I \to I$ is a 
linear map. The condition $a^{\circ 3} = 0$ implies $a^{\circ 2} \in \Ker f$, 
and the Jacobi identity for triple $(a,a,x)$, where $x\in I$, implies $f^2 = 0$.

If $V$ is a proper $f$-invariant subspace of $I$, then it is a proper ideal of 
$L$, and the quotient $L/V = I/V \dotplus Ka$ is abelian, i.e., 
$\im f \subseteq V$ and $a^{\circ 2} \in V$. So $f$ is a nilpotent of index 
$\le 2$ linear map such that its image lies in any proper invariant subspace, 
what implies that either $f=0$, or $f$ is a map on a $2$-dimensional vector 
space, represented by the matrix $\begin{pmatrix} 0 & 1 \\ 0 & 0 \end{pmatrix}$
in the canonical basis. Moreover, $a^{\circ 2}$ lies in any proper $f$-invariant
subspace, what implies that either $a^{\circ 2} = 0$, or $f$ is the zero map on
an one-dimensional vector space, or, again, $f$ is a map on a $2$-dimensional 
vector space, represented by the matrix 
$\begin{pmatrix} 0 & 1 \\ 0 & 0 \end{pmatrix}$ in the canonical basis.
Combining all these possibilities, we get the claimed list of algebras.
\end{proof}

\begin{proof}[First proof of Theorem \ref{th-pbw}]
Let us prove by induction on dimension of $L$, that for a nonabelian mock-Lie
algebra $L$, $\dim U(L) < 2^{\dim L}$. The only one-dimensional mock-Lie algebra
is abelian, so the statement is vacuously true for $\dim L = 1$. The induction 
step: if $L$ has either a nonabelian ideal, or nonabelian quotient, we are done
by Lemma \ref{lemma-ul}, otherwise $L$ is one of the algebras described in 
Lemma \ref{lemma-ab}, which can be dealt with directly (see also computer 
calculations below).
\end{proof}

Now we turn to the second proof.

\begin{lemma}\label{lemma-quadr}
For any commutative algebra $A$, one of the following holds:
\begin{enumerate}[\upshape(i)]
\item $A$ is algebra with trivial multiplication;
\item
$A$ is of the form $V \dotplus Ka$, where $V$ is a vector space with trivial 
multiplication, $a \circ x = x$ for any $x \in V$, and $a^{\circ 2} = 2a$;
\item
$A$ contains an element $x$ such that $x$ and $x^{\circ 2}$ are linearly 
independent.
\end{enumerate}
\end{lemma}

Note that the algebra $A$ here is not assumed to be mock-Lie, or Jordan, or 
to satisfy any other distinguished identity beyond commutativity. We continue to
denote the binary multiplication by $\circ$.

\begin{proof}
Suppose that for any element $x$ of the algebra $A$, $x$ and $x^{\circ 2}$ are 
linearly dependent, i.e., for any $x\in A$, there is 
$\lambda(x) \in K$ such that 
\begin{equation}\label{eq-l}
x^{\circ 2} = \lambda(x) x .
\end{equation}
Linearizing this equality, and using bilinearity and commutativity of $\circ$, 
we get that there exists a linear map $\alpha: A \to K$ such that 
$x \circ y = \alpha(y)x + \alpha(x)y$ for any $x,y\in A$. If $\alpha$ is the 
zero map, the multiplication in $A$ is trivial. Otherwise, set 
$V = \Ker \alpha$. Since $V$ is of codimension one in $A$, we may write 
$A = V \dotplus Ka$, and normalize $\alpha$ by assuming $\alpha(a) = 1$. The 
rest is obvious.
\end{proof}

Algebras with the condition that any element satisfies a cubic polynomial, or,
in other words, algebras of rank $\le 3$ (of which the condition to be of 
nil index $3$ is the particular case) were studied in a number of papers -- see,
for example \cite{walcher} and references therein. The condition (\ref{eq-l}) is
the quadratic particular case, i.e., it defines algebras of rank $\le 2$ 
(without the free term in the defining polynomial equation), and headings (i) 
and (ii) of Lemma \ref{lemma-quadr} give an easy description of such 
commutative algebras. Note that algebra in (ii) is Jordan, but not mock-Lie.

\begin{proof}[Second proof of Theorem \ref{th-pbw}]
Obviously, for abelian mock-Lie algebras the Poincar\'e--Birkhoff--Witt theorem
holds. Conversely, suppose $L$ is a nonabelian mock-Lie algebra of dimension 
$n$. By Lemma \ref{lemma-quadr}, there is $x_2\in L$ such that $x_2$ and 
$x_1 = x_2^{\circ 2}$ are linearly independent. Complete $x_1, x_2$ to a basis 
$\{x_1, x_2, x_3, \dots, x_n\}$ of $L$. Since in $U(L)$ holds $x_1 = x_2^2$, 
$U(L)$ is generated, as an algebra with unit, by $n-1$ elements 
$\{x_2, x_3, \dots, x_n\}$, and hence is linearly spanned by monomials of
the form $x_2^k x_{i_1} \dots x_{i_\ell}$, where $0 \le k \le 2$, 
$3 \le i_1 < \dots < i_\ell \le n$; consequently, 
$\dim U(L) \le 3 \cdot 2^{n-2} < 2^n$.
\end{proof}

In fact, the second proof ``almost'' establishes Theorem \ref{th-pbw} in the 
broader class of all Jordan algebras: the finite-dimensional Jordan algebras for
which the Poincar\'e--Birkhoff-Witt theorem holds, are exactly algebras with 
trivial multiplication, and algebras specified in Lemma \ref{lemma-quadr}(ii).

Direct calculations, performed with the aid of GAP \cite{gap} and GAP package 
GBNP \cite{gbnp}, show that, as a rule, the dimension of $U(L)$ is much smaller
than $2^{\dim L}$. In the table below, the left column indicates the name of a 
mock-Lie algebra, and the right column the dimension of its universal enveloping
algebra. The table contains all nonabelian algebras of dimension $\le 5$, and 
all nonassociative algebras of dimension $6$. We follow the nomenclature of 
\cite{burde-fial}, with the exception of $\mathfrak L$ which denotes the unique,
over an algebraically closed field, $5$-dimensional nonassociative mock-Lie 
algebra (see \cite[Proposition 4.1]{burde-fial}). In the family of 
$6$-dimensional algebras $A_{26}(\beta,0)$ and $A_{26}(\beta,1)$, the parameter
$\beta$ assumes values $0,1$.

\begin{center}
\begin{tabular}[t]{|l|r|}

\multicolumn{2}{}{}             \\
\multicolumn{2}{c}{dimension 2} \\
\hline
$A_{12}$ & 3  \\
\hline

\multicolumn{2}{}{}             \\
\multicolumn{2}{c}{dimension 3} \\
\hline
$A_{12} \oplus A_{01}$ & 5  \\
$A_{13}$               & 5  \\
\hline

\multicolumn{2}{}{}              \\
\multicolumn{2}{c}{dimension 4}  \\
\hline
$A_{14}$                             & 7  \\
$A_{24}$                             & 9  \\
$A_{12} \oplus A_{01} \oplus A_{01}$ & 9  \\
$A_{13} \oplus A_{01}$               & 9  \\
$A_{12} \oplus A_{12}$               & 6  \\
\hline

\end{tabular}
\qquad
\begin{tabular}[t]{|l|r|}

\multicolumn{2}{}{}             \\
\multicolumn{2}{c}{dimension 5} \\
\hline
$A_{15}$                                           &  9 \\
$A_{25}$                                           & 10 \\
$A_{35}$                                           & 10 \\
$A_{45}$                                           & 10 \\
$A_{55}$                                           & 10 \\
$A_{65}$                                           & 11 \\
$A_{75}$                                           & 17 \\
$A_{12} \oplus A_{01} \oplus A_{01} \oplus A_{01}$ & 17 \\
$A_{13} \oplus A_{01} \oplus A_{01}$               & 17 \\
$A_{24} \oplus A_{01}$                             & 17 \\
$A_{12} \oplus A_{12} \oplus A_{01}$               & 10 \\
$A_{12} \oplus A_{13}$                             & 10 \\
$A_{14} \oplus A_{01}$                             & 11 \\
$\mathfrak L$                                      &  9 \\      
\hline

\end{tabular}
\qquad
\begin{tabular}[t]{|l|r|}

\multicolumn{2}{}{}             \\
\multicolumn{2}{c}{dimension 6} \\
\hline
$A_{16}$           & 11  \\
$A_{26}(\beta,0)$  & 13  \\
$A_{26}(\beta,1)$  & 12  \\
\hline

\end{tabular}
\end{center}

\bigskip

Note also that, unlike in the Lie case, $U(L_1 \oplus L_2)$ is not isomorphic to
(in fact, in most of the cases much smaller than) $U(L_1) \otimes U(L_2)$.

\begin{theorem}\label{th-l4}
For any mock-Lie algebra $L$, the kernel of the map $\iota: L \to U(L)$ lies in
$L^4$.
\end{theorem}

\begin{proof}
Suppose that $z \in \Ker\iota$, i.e., $z$ belongs to the ideal of $T(L)$ 
generated by elements of the form (\ref{eq-id}). We may write
\begin{equation}\label{eq-eq}
z = 
\sum_{i\ge 0} \sum_{j\ge 0} \sum_{k\in \mathbb I_{ij}} 
a_i^{(k)} \otimes \Big(
\frac12(x_{ij}^{(k)} \otimes y_{ij}^{(k)} + y_{ij}^{(k)} \otimes x_{ij}^{(k)}) - x_{ij}^{(k)} \circ y_{ij}^{(k)}
\Big)
\otimes b_j^{(k)} ,
\end{equation}
where $a_i^{(k)}, b_i^{(k)}$ are homogeneous elements of $T(L)$ of degree $i$, 
and $x_{ij}^{(k)}, y_{ij}^{(k)} \in L$ (we identify elements 
$\lambda \otimes x$, $x \otimes \lambda$, and $\lambda x$ for 
$\lambda \in T^0(L) = K$ and $x\in T(L)$). By modifying $x$'s and $y$'s
appropriately, we may normalize the degree $0$ elements $a_0^{(k)}$ and 
$b_0^{(k)}$ by equating them to $1$. Moreover, by expanding the terms in 
further sums, if necessary, we may assume that $a$'s and $b$'s are monomials,
i.e., $a_i^{(k)} = a_i^{(k1)} \otimes \dots \otimes a_i^{(ki)}$, where
$a_i^{(k*)} \in L$, and similarly for $b$'s.

Isolating in (\ref{eq-eq}) the homogeneous components, we get
\begin{equation}\label{eq-z}
z = - \sum_{k\in \mathbb I_{00}} x_{00}^{(k)} \circ y_{00}^{(k)}
\end{equation}
in degree $1$ (i.e., for elements lying in $L$), and
\begin{equation}\label{eq-n}
\sum_{i+j=n-2} \> \sum_{k\in \mathbb I_{ij}} 
a_i^{(k)} \otimes
\frac12(x_{ij}^{(k)} \otimes y_{ij}^{(k)} + y_{ij}^{(k)} \otimes x_{ij}^{(k)}) 
\otimes b_j^{(k)} 
=
\sum_{i+j=n-1} \> \sum_{k\in \mathbb I_{ij}} 
a_i^{(k)} \otimes (x_{ij}^{(k)} \circ y_{ij}^{(k)}) \otimes b_j^{(k)}
\end{equation}
in degree $n>1$ (i.e., for elements lying in $L^{\otimes n}$). The equality 
(\ref{eq-z}) implies $z \in L^2$. 

Applying to both sides of the equality (\ref{eq-n}) for $n=2$ the multiplication
map $\circ: L \otimes L \to L$, and using commutativity of $L$, we get
\begin{equation}\label{eq-3}
\sum_{k\in \mathbb I_{00}} x_{00}^{(k)} \circ y_{00}^{(k)} 
=
  \sum_{k\in \mathbb I_{10}} (x_{10}^{(k)} \circ y_{10}^{(k)}) \circ a_1^{(k)}
+ \sum_{k\in \mathbb I_{01}} (x_{01}^{(k)} \circ y_{01}^{(k)}) \circ b_1^{(k)} ,
\end{equation}
what, together with (\ref{eq-z}), implies $z \in L^3$.

Applying to both sides of the equality (\ref{eq-n}) for $n=3$ the map 
$L \otimes L \otimes L \to L$, 
$a \otimes b \otimes c \mapsto (a \circ c) \circ b$, and taking into account 
commutativity and the Jacobi identity, we get
\begin{gather*}
\sum_{k\in \mathbb I_{10}} (x_{10}^{(k)} \circ y_{10}^{(k)}) \circ a_1^{(k)}
+ 
\sum_{k\in \mathbb I_{01}} (x_{01}^{(k)} \circ y_{01}^{(k)}) \circ b_1^{(k)} 
= \\
- 2\Big(
  \sum_{k\in \mathbb I_{02}} 
((x_{02}^{(k)} \circ y_{02}^{(k)}) \circ b_2^{(k2)}) \circ b_2^{(k1)}
+ \sum_{k\in \mathbb I_{11}} 
(x_{11}^{(k)} \circ y_{11}^{(k)}) \circ (a_1^{(k)} \circ b_1^{(k)})
+ \sum_{k\in \mathbb I_{20}} 
((x_{20}^{(k)} \circ y_{20}^{(k)}) \circ a_2^{(k1)}) \circ a_2^{(k2)}
\Big) ,
\end{gather*}
what, together with (\ref{eq-z}) and (\ref{eq-3}), implies $z \in L^4$.
\end{proof}

One may try to continue the same way, considering the equality (\ref{eq-n}) in 
higher degrees, and applying maps on the tensor powers of $L$ with different 
bracketings and permutations. For example, in the case $n=4$, applying to both sides of (\ref{eq-n}) the map 
$L \otimes L \otimes L \otimes L \to L$, 
$a \otimes b \otimes c \otimes d \mapsto ((a \circ b) \circ c) \circ d$, and 
using the Jacobi identity, we get:
\begin{gather*}
\sum_{k\in \mathbb I_{02}} 
((x_{02}^{(k)} \circ y_{02}^{(k)}) \circ b_2^{(k1)}) \circ b_2^{(k2)}
- \frac12 \sum_{k\in \mathbb I_{11}} 
((x_{11}^{(k)} \circ y_{11}^{(k)}) \circ a_1^{(k)}) \circ b_1^{(k)} 
- \frac12 \sum_{k\in \mathbb I_{20}} 
(x_{20}^{(k)} \circ y_{20}^{(k)}) \circ (a_2^{(k1)} \circ a_2^{(k2)})
\\=
\sum_{k\in \mathbb I_{03}} 
(((x_{03}^{(k)} \circ y_{03}^{(k)}) \circ b_3^{(k1)}) \circ b_3^{(k2)}) \circ
b_3^{(k3)}
+
\sum_{k\in \mathbb I_{12}} 
((a_1^{(k)} \circ (x_{12}^{(k)} \circ y_{12}^{(k)})) \circ b_2^{(k1)}) \circ
b_2^{(k2)}
\\+
\sum_{k\in \mathbb I_{21}} 
((a_2^{(k1)} \circ a_2^{(k2)}) \circ (x_{21}^{(k)} \circ y_{21}^{(k)})) \circ 
b_1^{(k)}
+
\sum_{k\in \mathbb I_{30}} 
((a_3^{(k1)} \circ a_3^{(k2)}) \circ a_3^{(k3)}) \circ (x_{30}^{(k)} \circ y_{30}^{(k)}) .
\end{gather*}
But the complexity of such manipulations explodes and it is not clear how
exactly to combine them to proceed. Moreover, since there exist 
finite-dimensional (and hence nilpotent) exceptional mock-Lie algebras 
(\S \ref{sec-nonsp}), it is impossible that $z \in L^n$ for any $n$, so this process should stop somewhere (actually, at 
$n \le 8$, as there is an exceptional mock-Lie algebra whose nilpotency index is
$9$).

\begin{proposition}
Any mock-Lie algebra of dimension $\le 6$ is special.
\end{proposition}

\begin{proof}[First proof]
According to \cite[Propositions 4.1 and 5.1]{burde-fial}, any nonassociative 
mock-Lie algebra $L$ of dimension $\le 6$ is isomorphic to one of the following
algebras: $\mathfrak L$ (of dimension 5), $A_{16}$, or $A_{26}(\beta,\delta)$ 
(of dimension 6). Using the same GAP/GBNP computations, we may check that 
Gr\"obner basis of the universal enveloping algebra of these algebras does not contain 
elements of degree $1$ (i.e., elements of $L$), hence the defining ideal of 
$U(L)$ does not contain such elements, and $L$ embeds into $U(L)$.
\end{proof}

\begin{proof}[Second proof]
For any nonassociative mock-Lie algebra $L$ of dimension $\le 6$, we have
$L^4 = 0$, hence according to Theorem \ref{th-l4}, $L$ embeds into $U(L)$.
\end{proof}

\begin{proof}[Third proof]
Follows from the result of Slin'ko \cite[Theorem 2]{slinko} that a nilpotent 
Jordan algebra of nilpotency index $\le 5$ is special.
\end{proof}

\section{No alternative route to Ado}\label{sec-lieado}

In \cite{ado}, an alternative proof of the Ado theorem for nilpotent Lie 
algebras was given, not utilizing universal enveloping algebras, but working 
entirely inside the category of finite-dimensional Lie algebras. Our initial 
attempt was to modify this proof for the mock-Lie case. This approach, however,
meets several obstacles, some of them seems difficult to surmount (and was 
doomed to failure anyway due to existence of exceptional mock-Lie algebras, see the next 
\S \ref{sec-nonsp}). We think it is instructive to examine these 
obstacles, in order to better understand the peculiarities of the mock-Lie case.

The proof in \cite{ado} goes as follows. First it is noted that an 
$\mathbb N_{< n}$-graded (i.e., $\mathbb N$-graded with nonzero components
concentrated in degrees $1,2,\dots,n-1$) Lie algebra $L$ can be embedded into 
its tensor product extension $L \otimes tK[t]/(t^n)$. This is carried over 
verbatim to the mock-Lie case. 

The associative commutative algebra $tK[t]/(t^n)$ possess a nondegenerate
derivation $t \frac{d}{dt}$. This allows to construct a faithful representation
of the Lie algebra $L \otimes tK[t]/(t^n)$ (and hence of its subalgebra $L$) of
the form 
$\big(L \otimes tK[t]/(t^n)\big) \ltimes 
\big(\id_L \otimes t \frac{d}{dt}\big)$. 
Completely similar with \cite[Lemma 1.3]{ado}, we have an elementary

\begin{lemma}\label{lemma-nondeg}
Let $L$ be a mock-Lie algebra, $V$ an $L$-module, and $D$ an antiderivation of 
$L$ with values in $V$ such that $\Ker D = 0$. Then $L$ has a faithful 
representation.
\end{lemma}

\begin{proof}
The required representation $\rho$ is given by the action of $L$ on 
$V \dotplus \Der_{-1}(L,V)$, defined naturally on the first direct summand, and
via $\rho(x)(d) = d(x)$ for $x\in L$ and $d \in \Der_{-1}(L,V)$, on the second direct summand.
\end{proof}

Now comes the first obstacle: unlike in the case of derivations, the algebra 
$t K[t]/(t^n)$ possess nondegenerate antiderivations if and only if $n\le 4$. 
More precisely, we have:

\begin{lemma}\label{lemma-antider}
The space of antiderivations of the algebra $tK[t]/(t^n)$ ($n\ge 2$) is 
$1$-dimensional for $n=2$, $2$-dimensional for $n=3$, and $3$-dimensional for
$n\ge 4$. Its basis can be chosen among antiderivations of the following
form (only nonzero actions on the standard basis $\{t,t^2,\dots,t^{n-1}\}$ of 
$tK[t]/(t^n)$ are given):
\begin{enumerate}[\upshape(i)]
\item $t \mapsto t^{n-1}$;
\item 
$t \mapsto -\frac 12 t^{n-2}, \quad t^2 \mapsto t^{n-1} \>$ (if $n \ge 3$);
\item 
$t \mapsto t^{n-3}, \quad 
 t^2 \mapsto -2t^{n-2}, \quad
 t^3 \mapsto t^{n-1} \>$ (if $n \ge 4$).
\end{enumerate}
\end{lemma}

\begin{proof}
Straightforward computation using induction on the degree of monomials.
\end{proof}

Thus we have only a very limited analog of \cite[Lemma 2.5]{ado}, with
$\mathbb N_{<4}$-gradings instead of arbitrary $\mathbb N$-gradings:

\begin{lemma}\label{lemma-grad}
An $\mathbb N_{<4}$-graded mock-Lie algebra has a faithful representation.
\end{lemma}

\begin{proof}
By above, an $\mathbb N_{<4}$-graded mock-Lie algebra $L$ is embedded into
$L \otimes tK[t]/(t^n)$, $n\le 4$. By Lemma \ref{lemma-antider}, 
$tK[t]/(t^n)$ has a nondegenerate antiderivation $D$, and hence 
$L \otimes tK[t]/(t^n)$ has a nondegenerate antiderivation $\id_L \otimes D$. 
Applying Lemma \ref{lemma-nondeg} to the adjoint module $L \otimes tK[t]/(t^n)$
and the antiderivation $\id_L \otimes D$, we get that $L \otimes tK[t]/(t^n)$ 
has a faithful representation, and hence so does its subalgebra $L$.
\end{proof}

Note that, unlike in the Lie case, the semidirect sum 
$\big(L \otimes tK[t]/(t^n)\big) \ltimes KD$ is not a mock-Lie algebra (for 
example, it is not locally nilpotent). We merely consider it as an
$L \otimes tK[t]/(t^n)$-module, ignoring its algebra structure.

The proof in \cite{ado} then proceeds by induction on $\dim I$ in the 
representation of an arbitrary nilpotent Lie algebra as a quotient $F/I$ of a 
free nilpotent Lie algebra $F$. The key ingredient in this induction is an
auxiliary result about possibility to distinguish elements of a Lie algebra by
the kernel of a suitable representation (\cite[Lemma 2.10]{ado}), which is
proved using some combinatorics related to the tensor product of 
representations. And here comes another obstacle: in general, there is no notion
of the tensor product of two representations of a mock-Lie algebra. Perhaps, one
may try to work around it by defining an ad-hoc bialgebra structure on mock-Lie
algebras in question using central elements, similarly how it is done in 
\cite[Propositions 1 and 2]{zhelyabin} in some particular cases of Jordan 
algebras, and get in this way that any mock-Lie algebra of nilpotency index 
$\le 4$ has a faithful representation, and hence is special. This result, 
however, would be covered by Theorem \ref{th-l4}, or, more generally, by the above-cited Slin'ko's result about speciality of nilpotent Jordan 
algebras of nilpotency index $\le 5$.

\section{Exceptional mock-Lie algebras}\label{sec-nonsp}

In \cite{hjs}, a $44$-dimensional mock-Lie algebra was presented, on which the
Glennie identity $G_8 = 0$ does not vanish (see \cite{mccrimmon} for a nice 
overview of the Glennie and other special Jordan identities), and hence it is 
exceptional as a Jordan algebra. Though this algebra was constructed using 
Albert \cite{albert-prog}, the authors meticulously define the multiplication 
table on the first $3$ pages of their preprint, and prove that it indeed defines
a mock-Lie algebra with non-vanishing Glennie identity on the next 15 pages.
Here we recreate this 20-years-old effort, indicating how one can rigorously 
establish existence of non-special mock-Lie algebras using computer -- in two,
somewhat different, ways. 

In order to verify whether a certain identity holds in a certain variety of 
algebras, Albert constructs a large enough (but finite-dimensional) homomorphic
image of a free algebra in a given variety (see \cite{jacobs} and Albert User's
Guide at \cite{albert-prog}). Thus, in order to verify whether the Glennie 
identity $G_8 = 0$, where $G_8$ is a word of the total degree $8$ in $3$ 
variables -- two of degree $3$ and one in degree $2$, holds in the variety of 
mock-Lie algebras, one constructs a quotient of the free mock-Lie algebra of rank $3$, freely generated by elements, say, $a$, $b$, $c$,
by the ideal linearly spanned by all words containing either at least $4$ $a$'s,
or at least $4$ $b$'s, or at least $3$ $c$'s. This quotient, let us call it 
$\mathfrak M$, has dimension $44$. Initially, Albert worked over the prime 
fields of characteristic $\le 251$, and we have modified it to work over the 
rationals (\cite{albert-mod}). When computing over the rationals, all nonzero 
coefficients in the multiplication table of $\mathfrak M$ belong to the set 
$\{\pm 2, \pm 1, \pm \frac 12\}$. This means that $\mathfrak M$ can be defined 
over any field of characteristic $\ne 2$. 

For the sake of further discussion, let $\{e_1, e_2, e_3, \dots, e_{44}\}$ be 
the standard basis of $\mathfrak M$ as produced by Albert. A quick inspection of
the multiplication table of $\mathfrak M$, supported by simple computer 
calculations, reveals that $\mathfrak M^9 = 0$, and the center of $\mathfrak M$ 
is one-dimensional, linearly spanned by $e_{44}$.

After constructing the multiplication table of a suitable homomorphic image of 
the free algebra, Albert checks whether a given identity $f = 0$ is satisfied in
that image, by computing the linearization of $f$ on all possible combinations 
of the basic elements, and verifying whether the results are identically zero.
Thus, we can check that $G_8 = 0$ is not an identity in the mock-Lie variety
for any fixed characteristic (well, technically for any characteristic 
$< 2^{63}$, what is the current limitation of \cite{albert-mod}) including zero, but, of course, we want to establish this for \emph{any} characteristic 
$\ne 2,3$. When the fact that some set of identities $f_1 = 0, \dots, f_n = 0$ 
does not imply another identity $f = 0$ in characteristic zero, implies the same
fact in (almost) any positive characteristic $p$? Using the standard 
ultraproduct argument, it is true for all $p > p_0$, for some $p_0$ depending on
$f_1, \dots, f_n, f$. However, this argument does not give any concrete value of
$p_0$. One can try to employ a slightly different method -- instead of 
verifying identity in the given variety, one can add the identity in question to
the defining identities, and to compare the dimension sequences of the 
corresponding pieces of free algebras produced by Albert. With such approach, 
using a simple variant of the Chinese remainder theorem, one can give a concrete
estimate on $p_0$ (see \cite{alternative}), however this estimate is too big for
all practical purposes. 

Instead, we can perform the last Albert's step -- verifying of (non-)identity --
in another, general purpose computer algebra system of our choice (we prefer GAP
\cite{gap} for such sort of tasks), which gives us finer control over the whole
process. Namely, we can verify that in the algebra $\mathfrak M$ over the 
rationals, $G_8(e_1, e_2, e_3) = 96e_{44}$. Obviously, this computation implies 
$G_8(e_1, e_2, e_3) \ne 0$ in $\mathfrak M$ modulo any prime $p \ne 2,3$.

A second way to verify that a given mock-Lie algebra $L$ is exceptional, is to 
compute a Gr\"obner basis of its universal enveloping algebra $U(L)$, and to 
verify whether this Gr\"obner basis contains elements of the first degree (i.e.,
elements of $L$). Indeed, due to universal property of $U(L)$, $L$ is special if
and only if it admits embedding into $U(L)$, what, in its turn, happens if and 
only if the defining ideal of $U(L)$, and hence any of the Gr\"obner bases of 
$U(L)$, does not contain elements of the first degree. Computing with the help 
of GAP and GBNP, as in \S \ref{sec-uj}, we see that the central element $e_{44}$
belongs to a Gr\"obner basis of $U(\mathfrak M)$. (By the way, 
$\dim U(\mathfrak M) = 157$).

The same procedures can be repeated for other known special Jordan identities,
but as all of them are of degree $>8$, the resulting algebras constructed by 
Albert are of higher dimensions. Among those special identities summarized in 
\cite{mccrimmon}, the identities $G_9$ and $S_9$ give rise to exceptional 
mock-Lie algebras of dimension $52$ and $177$ respectively, and the rest of 
identities are, within the mock-Lie variety, consequences of $G_8$, so they do 
not produce a new algebra. The Medvedev special identity (\cite[\S 2]{medvedev})
gives rise to an exceptional mock-Lie algebra of dimension $144$.

\begin{question}\label{quest-dim}
What is the minimal possible dimension of an exceptional mock-Lie algebra?
\end{question}

According to Proposition in \S \ref{sec-uj}, this minimal possible dimension 
lies between $7$ and $44$. 

The algebra $\mathfrak M$ has another interesting feature: since the quotient of
any mock-Lie algebra by the center is special, $\mathfrak M$ is an 
one-dimensional central extension of the $43$-dimensional special algebra. The 
analogous Lie-algebraic situation (where, alas, speciality is understood as an 
embedding into associative \emph{PI} algebras) was discussed many times in the 
literature, with the decisive example of a special Lie algebra whose 
one-dimensional central extension is not special, albeit, naturally, of a very 
different sort (\cite{billig-hom} and references therein).

Moreover, experimenting with GAP, we have found that all random subalgebras and
random quotients of $\mathfrak M$ we have tried, are special, what inclines us 
to think that the minimal dimension in Question \ref{quest-dim} is indeed $44$.

\begin{question}
What is the minimal nilpotency index of an exceptional mock-Lie algebra? 
\end{question}

According to the above-cited Slin'ko's theorem, for Jordan algebras this minimal
nilpotency index is $6$, so for mock-Lie algebras it lies between $6$ and 
$9$. Of course, the Glennie identity, being of degree $8$, is satisfied in all 
mock-Lie algebras of nilpotency index $8$, what, however, still does not 
guarantee that all of them are special.

\section*{Acknowledgements}

Thanks are due to Yuly Billig, Vladimir Dotsenko, Viktor Zhelyabin, and 
especially Ivan Shestakov for useful remarks. Ivan Shestakov has put me on the 
right track by pointing out that almost all statements and conjectures in the 
first draft of the paper were wrong, giving idea of the proof of 
Theorem \ref{th-pbw}, and informing about relevant unpublished preprints, among them \cite{hjs} and \cite{sverchkov-prep}.
This work was supported by 
the Statutory City of Ostrava (grant 0924/2016/Sa\v{S}), and
the Ministry of Education and Science of the Republic of Kazakhstan 
(grant 0828/GF4). 
Some computations were performed using facilities of the National 
Supercomputing Center at the Technical University of Ostrava, supported by 
the Ministry of Education, Youth and Sports of Czech Republic (grant LM2015070).

\renewcommand{\refname}{Software}

\end{document}